\pdfoutput=1
\RequirePackage{ifpdf}
\ifpdf 
\documentclass[pdftex]{sigma}
\else
\documentclass{sigma}
\fi

\numberwithin{equation}{section}

\newtheorem{Theorem}{Theorem}[section]
\newtheorem*{Theorem*}{Theorem}

\newtheorem{Proposition}[Theorem]{Proposition}
 { \theoremstyle{definition}
\newtheorem{Definition}[Theorem]{Definition}

\newtheorem{Example}[Theorem]{Example}
\newtheorem{Remark}[Theorem]{Remark} }

\begin{document}
\allowdisplaybreaks

\newcommand{\arXivNumber}{2401.03438}

\renewcommand{\PaperNumber}{042}

\FirstPageHeading

\ShortArticleName{Asymptotic Expansions of Finite Hankel Transforms}

\ArticleName{Asymptotic Expansions of Finite Hankel Transforms\\ and the Surjectivity of Convolution Operators}

\Author{Yasunori OKADA~$^{\rm a}$ and Hideshi YAMANE~$^{\rm b}$}

\AuthorNameForHeading{Y.~Okada and H.~Yamane}

\Address{$^{\rm a)}$~Graduate School of Science, Chiba University, \\
\hphantom{$^{\rm a)}$}~Yayoicho 1-33, Inage-ku, Chiba, 263-8522, Japan}
\EmailD{\href{mailto:okada@math.s.chiba-u.ac.jp}{okada@math.s.chiba-u.ac.jp}}
\URLaddressD{\url{https://www.math.s.chiba-u.ac.jp/~okada/}}

\Address{$^{\rm b)}$~Department of Mathematical Sciences, Kwansei Gakuin University, \\
\hphantom{$^{\rm b)}$}~Uegahara, Gakuen, Sanda, Hyogo, 669-1330, Japan}
\EmailD{\href{mailto:yamane@kwansei.ac.jp}{yamane@kwansei.ac.jp}}
\URLaddressD{\url{https://sci-tech.ksc.kwansei.ac.jp/~yamane/}}

\ArticleDates{Received January 10, 2024, in final form May 21, 2024; Published online May 27, 2024}

\Abstract{A compactly supported distribution is called invertible in the sense of Ehren\-preis--{H\"ormander} if the convolution with it induces a~surjection from $\mathcal{C}^{\infty}(\mathbb{R}^{n})$ to itself. We give sufficient conditions for radial functions to be invertible. Our analysis is based on the asymptotic expansions of finite Hankel transforms. The dominant term may be the contribution from the origin or from the boundary of the support of the function. For the proof, we propose a new method to calculate the asymptotic expansions of finite Hankel transforms of functions with singularities at a point other than the origin.}

\Keywords{convolution; asymptotic expansion; Hankel transform; invertibility}

\Classification{45E10; 33C10; 44A15}

\section{Introduction}

It is known that
\[P(D)\colon\ \mathcal{C}^{\infty}(\mathbb{R}^{n})\to\mathcal{C}^{\infty}(\mathbb{R}^{n})\qquad
\text{and}\qquad P(D)\colon\ \mathcal{D}'(\mathbb{R}^{n})\to\mathcal{D}'(\mathbb{R}^{n})
\]
are both surjective, where $P(D)\ne0$ is an arbitrary linear partial
differential operator with constant coefficients~\mbox{\cite{Ehrenpreis54,Ehrenpreis56,Malgrange}}.
Since these mappings coincide with the convolution operator~$P(D)\delta*$
on respective spaces, a natural question is to characterize compactly
supported distributions that induce surjective convolution operators.
Such a distribution is called invertible. Notice that~${\delta(x-a)}$
induces translation and hence is invertible. It was found that a compactly
supported distribution is invertible if and only if its Fourier transform,
an entire function, is slowly decreasing in a~certain sense. This
condition was found by Ehrenpreis~\cite{Ehrenpreis60} and was further
studied by H\"ormander \cite{Hormander62}. A self-contained account
can be found in \cite{Hormander2}. Invertible distributions and slowly
decreasing entire functions became and still are fundamental concepts
in the research of convolution equations. Perturbation of invertible
distributions are studied in \cite{Musin}, and a recent paper \cite{DK}
investigates invertible distributions in an abstract setting. Convolution
equations on symmetric spaces are discussed from the viewpoint of
invertibility in \cite{Kakehi}.

Slow decrease is a technical estimate from below and is not easy to
grasp. The present authors expect that the notions of slow decrease
and of invertibility will become less mysterious if many examples
or sufficient conditions are found. The delta function $\delta_{S(0,r)}$
supported on a~sphere is invertible~\cite{Lim}. Its normal derivatives
are invertible as well~\cite{OY_Tsukuba}. In \cite{Lim,OY_Tsukuba}, the asymptotic behavior of Bessel functions were
used. As is well known~\cite{SW}, the Fourier transform of a radial
function can be written in terms of an integral involving a Bessel
function. The Fourier transforms of~$\delta_{S(0,r)}$ and its normal
derivatives are limit cases.

In the present article, we prove the invertibility of some radial
functions by using two methods of calculating asymptotic expansions
of finite Hankel transforms $\int_{0}^{1}\varphi(s)J_{\nu}(rs){\rm d}s$.
One is a result in Wong \cite{Wong1976}, in which singularities at
$s=0$ determine the asymptotic behaviors. The other method, devised
by the present authors, is useful to deal with singularities at $s=1$.
As far as they know, this is the first result about the asymptotic
behaviors of Hankel transforms of functions singular at a point other
than $s=0$.

Notice that various additive formulas for supports and singular supports
of convolutions and spaces of entire functions with a certain type
of slow decrease are investigated in \cite{BD} and applied to study
the surjectivity of convolution operators. Several examples are discussed
in it and the characteristic function of an ellipsoid is studied by
using a Bessel function. The tools developed in the present paper
could be useful in that line of research.

\section{Invertibility and slow decrease}

In this article, we follow the convention
\[
\hat{f}(\xi)=\big\langle f(x),{\rm e}^{-{\rm i}\xi\cdot x}\big\rangle=\int_{\mathbb{R}^{n}}{\rm e}^{-{\rm i}\xi\cdot x}f(x){\rm d}x,\qquad \xi\in\mathbb{R}^{n}
\]
for the definition of the Fourier transform of a compactly supported
distribution $f\in\mathcal{E}'(\mathbb{R}^{n})$. When $f(x)\in L^{1}(\mathbb{R}^{n})$
is a radial function, its Fourier transform can be written in terms
of an integral involving a Bessel function.
\begin{Theorem}[{\cite[p.~155]{SW}}]
\label{thm:Stein-Weiss} Let $f_{0}(s)$, $s>0$, be a function of
a single variable. The Fourier transform of $f(x)=f_{0}(|x|)$, $x\in\mathbb{R}^{n}$,
is written in terms of a Hankel transform. More precisely,
\[
 \hat{f}(\xi)=\frac{(2\pi)^{n/2}}{r^{n/2-1}}\int_{0}^{\infty}s^{n/2}f_{0}(s)J_{n/2-1}(rs){\rm d}s,\qquad r=|\xi|,\qquad\xi\in\mathbb{R}^{n}.
\]
Here $J_{n/2-1}(z)$ denotes the Bessel function of the first kind
of order $n/2-1$.
\end{Theorem}

\begin{Remark}
Notice that $z^{-(n/2-1)}J_{n/2-1}(sz)$ is an \textit{even} entire
function in a single variable $z$ for each $s>0$. If the support
of $f_{0}(s)$ is bounded, then the Fourier transform of $f(x)=f_{0}(|x|)$
is an even entire function which is expressed by
\[
\hat{f}(\zeta)=(2\pi)^{n/2}\int_{0}^{\infty}s^{n/2}f_{0}(s)\bigl(\sqrt{\zeta^{2}}\bigr)^{-(n/2-1)}J_{n/2-1}\bigl(s\sqrt{\zeta^{2}}\bigr){\rm d}s,
\]
 where $\zeta=(\zeta_{1},\dots,\zeta_{n})\in\mathbb{C}^{n}$, $\zeta^{2}=\sum_{j=1}^{n}\zeta_{j}^{2}$.
\end{Remark}

We employ the notions of invertibility and slow decrease originated
by Ehrenpreis \cite{Ehrenpreis60} and refined by H\"ormander \cite{Hormander62}.
\begin{Theorem}[{\cite[Theorem 2.2, Proposition 2.7]{Ehrenpreis60}, \cite[Definition 3.1, Corollary 3.1]{Hormander62},
\cite[Theorems 16.3.9 and 16.3.10]{Hormander2}}] \label{thm:invertible}
For $u\in\mathcal{E}'(\mathbb{R}^{n})$, the following statements
are equivalent.
\begin{itemize}\itemsep=0pt
\item[$(i)$] There is a constant $A>0$ such that we have
\[
\sup\bigl\{ |\hat{u}(\zeta)|; \,\zeta\in\mathbb{C}^{n},\, |\zeta-\xi|<A\log(2+|\xi|)\bigr\} >(A+|\xi|)^{-A}\label{eq:0706d}
\]
for any $\xi\in\mathbb{R}^{n}$.
\item[$(ii)$] There is a constant $A>0$ such that we have
\begin{equation}
\sup\bigl\{ |\hat{u}(\eta)|;\, \eta\in\mathbb{R}^{n},\,|\eta-\xi|<A\log(2+|\xi|)\bigr\} >(A+|\xi|)^{-A}\label{eq:0706d-1}
\end{equation}
for any $\xi\in\mathbb{R}^{n}$.
\item[$(iii)$] If $w\in\mathcal{E}'(\mathbb{R}^{n})$ and $\hat{w}/\hat{u}$
is a~holomorphic function, then $\hat{w}/\hat{u}$ is the Fourier
transform of a~distribution in $\mathcal{E}'(\mathbb{R}^{n})$.
\item[$(iv)$] If $v\in\mathcal{E}'(\mathbb{R}^{n})$ satisfies $u*v\in\mathcal{\mathcal{C}}^{\infty}(\mathbb{R}^{n})$,
then $v\in\mathcal{\mathcal{C}}^{\infty}(\mathbb{R}^{n})$.
\item[$(v)$] The mapping $u*\colon\mathcal{\mathcal{C}}^{\infty}(\mathbb{R}^{n})\to\mathcal{\mathcal{C}}^{\infty}(\mathbb{R}^{n})$
is surjective.
\item[$(vi)$] The mapping $u*\colon\mathcal{D}'(\mathbb{R}^{n})\to\mathcal{D}'(\mathbb{R}^{n})$
is surjective.
\end{itemize}
\end{Theorem}

\begin{Remark}
In the present paper, we only deal with surjectivity on $\mathbb{R}^{n}$.
If one wants to study surjectivity on subsets of $\mathbb{R}^{n}$,
one needs the notions of $\mu$-convexity for supports or singular
supports~\cite{Hormander2}.
\end{Remark}

\begin{Definition}[{\cite[Definition 16.3.12]{Hormander2}}]\label{def:invertible} An element $u(x)$ of $\mathcal{E}'(\mathbb{R}^{n})$
is called \textit{invertible }and its Fourier transform $\hat{u}(\zeta)$
is called \textit{slowly decreasing} if the equivalent conditions
in Theorem \ref{thm:invertible} are fulfilled. See Definition \ref{def:minor extension}
below.
\end{Definition}

\begin{Remark}
\label{rem:invertible distributions}A finitely supported non-zero
distribution is invertible \cite[Theorem 4.4]{Hormander62}. The
most important example is $P(D)\delta(x)$, where $P(D)\ne0$ is a
linear partial differential operator with constant coefficients. It
give rise to $P(D)\delta(x)*$, which is nothing but $P(D)$. So $P(D)\colon\mathcal{D}'(\mathbb{R}^{n})\to\mathcal{D}'(\mathbb{R}^{n})$
and $P(D)\colon\mathcal{C}^{\infty}(\mathbb{R}^{n})\to\mathcal{C}^{\infty}(\mathbb{R}^{n})$
are surjective.

Other classes of invertible distributions are discussed in \cite{Abramczuk}.
Let $u$ be a compactly supported measure with an atom, $v\in\mathcal{E}'(\mathbb{R}^{n})$
have singular support disjoint from that of $u$ and $P(D)$ be a
non-zero linear partial differential operator with constant coefficients.
Then $P(D)u+v$ is invertible.

Another useful fact in \cite{Abramczuk} is the following. Let $u\in\mathcal{E}'(\mathbb{R}^{n})$,
$f\in\mathcal{C}^{\infty}(\mathbb{R}^{n})$. If $f$ is real analytic
in a neighborhood of $\operatorname{singsupp} u$ and $fu$ is invertible,
then $u$ is invertible. In \cite{Lim,OY_Tsukuba}, the
spherical mean value operator and its variants are discussed from
the view point of invertibility.
\end{Remark}

\begin{Proposition}\quad
\label{prop:u*v}
\begin{itemize}\itemsep=0pt
\item[$(i)$] If $u,v\in\mathcal{E}'(\mathbb{R}^{n})$ are invertible, then
so is $u*v$.
\item[$(ii)$] Let $\alpha\in\mathbb{R}\setminus \{ 0 \} $. If $u(x)\in\mathcal{E}'(\mathbb{R}^{n})$
is invertible, so is $u(\alpha x)$.
\item[$(iii)$] Let $a\in\mathbb{R}^{n}$. If $u(x)\in\mathcal{E}'(\mathbb{R}^{n})$
is invertible, so is $u(x-a)$.
\item[$(iv)$] If $u(x)\in\mathcal{E}'(\mathbb{R}^{n})$ is invertible, then
so is $\sum_{j=1}^{J}P_{j}(D)u(x-a_{j})$, where $P_{j}(D)$ is a
non-zero linear partial differential operator with constant coefficients
and $a_{j}\in\mathbb{R}^{n}$.
\item[$(v)$] Let $u\in\mathcal{E}'(\mathbb{R}^{n})$ and $v\in\mathcal{C}_{0}^{\infty}(\mathbb{R}^{n})$.
Then, $u+v$ is invertible if and only if $u$ is.
\item[$(vi)$] If $u(x)\in\mathcal{E}'(\mathbb{R}^{m})$ and $v(x')\in\mathcal{E}'(\mathbb{R}^{n})$
are invertible, then so is $(u\otimes v)(x,x')=u(x)v(x')\in\mathcal{E}'(\mathbb{R}^{m+n})$.\footnote{This can be proved in an alternative manner using general facts about
the tensor product of nuclear Fr\'echet spaces.}
\end{itemize}
\end{Proposition}

\begin{proof}
The proofs of (i)--(iv) are easy. Theorem \ref{thm:invertible}\,(ii)
and the Paley--Wiener--Schwartz theorem imply~(v). Finally, (vi) follows
from Theorem \ref{thm:invertible}\,(ii). Indeed, if~$u$ and~$v$
satisfy \eqref{eq:0706d-1} with a~common constant~$A$, then $u\otimes v$
satisfies \eqref{eq:0706d-1} with $2A$ instead of~$A$.
\end{proof}

\begin{Definition}\label{def:minor extension}In the present paper, we extend the terminology
of Definition \ref{def:invertible} slightly: an entire function $p(\zeta)$,
not necessarily the Fourier transform of a compactly supported distribution,
is called \textit{slowly decreasing} if there is a constant $A>0$
such that we have
\begin{equation}
\sup\bigl\{ |p(\eta)|;\, \eta\in\mathbb{R}^{n},\,|\eta-\xi|<A\log(2+|\xi|)\bigr\} >(A+|\xi|)^{-A}\label{eq:slowly decreasing f}
\end{equation}
for any $\xi\in\mathbb{R}^{n}$.
\end{Definition}

\begin{Proposition}\label{prop:slowly decreasing p}Let $p(\zeta)$ be an entire function.
There are constants $A,B>0$ such that
\begin{equation}
\sup\bigl\{ |p(\eta)|;\, \eta\in\mathbb{R}^{n},\,|\eta-\xi|<A\log(2+|\xi|)\bigr\} >(A+|\xi|)^{-A}\label{eq:1018a}
\end{equation}
for any $\xi\in\mathbb{R}^{n}$ satisfying $|\xi|\ge B$, if and only
if $p(\zeta)$ is slowly decreasing in the sense of Definition {\rm\ref{def:minor extension}}.
\end{Proposition}

\begin{proof}
The ``if'' part is trivial. We prove the ``only if'' part. If
$|\xi_{0}|=B$, we set
\[
S(\xi_{0})=\bigl\{ \eta\in\mathbb{R}^{n}; \,|\eta-\xi_{0}|<A\log(2+B)\bigr\} .
\]
By \eqref{eq:1018a},
\begin{equation}
\sup\bigl\{ |p(\eta)|;\,\eta\in S(\xi_{0})\bigr\} >(A+B)^{-A}.\label{eq:0706a}
\end{equation}
On the other hand, we set $
S= \{ \eta\in\mathbb{R}^{n};\,|\eta|<B+A\log(2+B) \} \supset S(\xi_{0})$.
Here we may assume~$A$ is so large that $A\log(2+B)>B.$ Then
\begin{equation}
\bigcup_{|\xi_{0}|=B}S(\xi_{0})=S.\label{eq:0706b}
\end{equation}
By \eqref{eq:0706a} and \eqref{eq:0706b}, we have
\begin{equation}
\sup \{ |p(\eta)|;\,\eta\in S \} >(A+B)^{-A}.\label{eq:0706c}
\end{equation}
Assume $|\xi|\le B$ and set
\[
S_{\xi}=\left\{ \eta\in\mathbb{R}^{n};\, |\eta-\xi|\le\frac{2B+A\log(2+B)}{\log2}\log(2+|\xi|)\right\} .
\]
Since
\begin{align*}
S \subset\bigl\{ \eta\in\mathbb{R}^{n};\,|\eta-\xi|<2B+A\log(2+B)\bigr\} \subset S_{\xi},
\end{align*}
\eqref{eq:0706c} implies $
\sup \{ |p(\eta);\, \eta\in S_{\xi} \} >(A+B)^{-A}>(A+B+|\xi|)^{-(A+B)}$.
We set
\[
\tilde{A}=\max\left\{ \frac{2B+A\log(2+B)}{\log2},\, A+B\right\} ,
\]
then \eqref{eq:slowly decreasing f} holds if we replace $A$ with~$\tilde{A}$.
\end{proof}

\begin{Proposition}\label{prop:single variable}Let $p(z)$ be an even entire function
of a single variable. If $p(z)$ is slowly decreasing, then $p\bigl(\sqrt{\zeta^{2}}\bigr)$
is a slowly decreasing function in $\zeta\in\mathbb{C}^{n}$.
\end{Proposition}

\begin{proof}Notice that $p\bigl(\sqrt{\zeta^{2}}\bigr)$ is well-defined since
$p(z)$ is even. There are constants $A,B>0$ such that we have
\begin{equation}
\sup\bigl\{ |p(y)|;\, y\in\mathbb{R},\,|y-x|<A\log(2+x)\bigr\} >(A+x){}^{-A}\label{eq:slowly decreasing in R}
\end{equation}
for any $x\in\mathbb{R}$ satisfying $x\ge B$. For any $\xi\in\mathbb{R}^{n}$,
set $x=|\xi|$. The radial function $p (|\eta| )$, $\eta\in\mathbb{R}^{n}$,
satisfies the condition in Proposition~\ref{prop:slowly decreasing p}, since the combination of
\begin{align*}
 \bigl\{ p(|\eta|);\,\eta\in\mathbb{\mathbb{R}}^{n},\,|\eta-\xi|<A\log(2+|\xi|)\bigr\}
 & =\bigl\{ p(|\eta|);\,\eta\in\mathbb{R}^{n},\,||\eta|-|\xi||<A\log(2+|\xi|)\bigr\} \\
 & =\{ p(y);\,y\in\mathbb{R},\,|y-x|<A\log(2+x)\}
\end{align*}
and \eqref{eq:slowly decreasing in R} implies
\begin{align*}
 \sup\bigl\{ p(|\eta|);\, \eta\in\mathbb{\mathbb{R}}^{n},\, |\eta-\xi|<A\log(2+|\xi|)\bigr\} >(A+x)^{-A}=(A+|\xi|)^{-A}
\end{align*}
if $x=|\xi|\ge B$.
\end{proof}
\begin{Proposition}
\label{prop:q single variable}Let $p(z)$ be an even entire function
of a single variable. Set $q(x)=x^{\alpha}p(x)$ for $x>0$, $\alpha\ge0$.
Assume there is a sufficiently small constant $C>0$ and sufficiently
large constants $A,B>0$ such that we have
\begin{equation}
\sup \{ |q(y)| ; \, y>0,\, |y-x|<B \} >Cx{}^{-A}\label{eq:slowly decreasing q}
\end{equation}
for any $x\ge B$. Then $p\bigl(\sqrt{\zeta^{2}}\bigr)$ is a slowly
decreasing entire function in $\zeta=(\zeta_{1},\dots,\zeta_{n})\in\mathbb{C}^{n}$,
where $\zeta^{2}=\sum_{j=1}^{n}\zeta_{j}^{2}$.
\end{Proposition}

\begin{proof}
The assertion is trivial when $\alpha=0$ (see Proposition \ref{prop:single variable}).
We have only to prove the case~${\alpha>0}$. By choosing a larger $B$
if necessary, we may assume that
\begin{equation}
2^{-\alpha}C>B^{-\alpha}.\label{eq:0706e}
\end{equation}
If $y>0$, $|y-x|<B$, we have $y^{\alpha}<(x+B)^{\alpha}<(2x)^{\alpha}$
and
\begin{equation}
|p(y)|=|y^{-\alpha}q(y)|>(2x)^{-\alpha}|q(y)|.\label{eq:0706f}
\end{equation}
By \eqref{eq:slowly decreasing q}, \eqref{eq:0706e} and \eqref{eq:0706f},
we have
\begin{align*}
 \sup\{ |p(y)|;\, y>0, \,|y-x|<B\}
 & >(2x)^{-\alpha}Cx^{-A}>B^{-\alpha}x^{-(A+\alpha)}\\
 & >x^{-(A+2\alpha)}>(A+2\alpha+x)^{-(A+2\alpha)}.
\end{align*}
On the other hand, $B<A\log(2+x)$ holds if $B/\log2\le A$. We see
that $p(z)$ is slowly decreasing since~\eqref{eq:slowly decreasing in R}
is valid if we adopt $\max\{ A+2\alpha,B/\log2\} $ as
a new value of $A$. Apply Proposition~\ref{prop:single variable}
to complete the proof.
\end{proof}

\section{Wong's result\label{sec:Wong's-result}}

We review the main result of \cite{Wong1976}. We incorporate the
minor correction given in \cite{Wong80correction}. Our notation is
different from that in \cite{Wong1976}. In particular, we employ
$s^{\mu+k}\,(k=0,1,2,\dots)$ instead of~$t^{s+\lambda-1}\,(s=0,1,2,\dots)$.
Let $\varphi(s)$ be a function in $s>0$ and suppose that $\int_{0}^{\infty}\varphi(s)J_{\nu}(rs){\rm d}s$
converges uniformly for all large values of $r$. We assume $\varphi(s)$
has the following three properties.
\begin{description}\itemsep=0pt
\item [{$(\Phi_{1})$}] $\varphi^{(m)}(s)$ is continuous in $s>0$, where
$m$ is a nonnegative integer.
\item [{$(\Phi_{2})$}] ${\varphi^{(j)}(s)\sim\sum_{k=0}^{\infty}c_{k}\frac{{\rm d}^{j}}{{\rm d}s^{j}}s^{\mu+k}}$ $(s\to+0$; $j=0,1,2,\dots,m)$,
where $c_{0}\ne0$, $\operatorname{Re}(\mu+\nu)>-1$, $m\ge\operatorname{Re}\mu+1$.
\item [{$(\Phi_{3})$}] $\int_{1}^{\infty}s^{-1/2}\varphi(s){\rm e}^{{\rm i}rs}{\rm d}s$, $\int_{1}^{\infty}s^{-1/2}s^{j-m-1/2}\varphi^{(j)}(s){\rm e}^{{\rm i}rs}{\rm d}s$ $(1\le j\le m)$
converges uniformly for all large values of $r$.
\end{description}
\begin{Theorem}[{\cite[equations~(3.6) and (4.2)]{Wong1976}, \cite[p.~409]{Wong80correction}}]
Assume that $(\Phi_{1})$--$(\Phi_{3})$ hold, and let $n$ be a
positive integer satisfying $m-\operatorname{Re}\mu-1<n<m-\operatorname{Re}\mu+\frac{1}{2}$.
Then
\begin{equation}
\int_{0}^{\infty}\varphi(s)J_{\nu}(rs){\rm d}s=\sum_{k=0}^{n-1}c_{k}\frac{\Gamma\left(\frac{1}{2}(\mu+k+\nu+1)\right)}{\Gamma\left(-\frac{1}{2}(\mu+k-\nu-1)\right)}\frac{2^{\mu+k}}{r^{\mu+k+1}} +o\bigl(r^{-m}\bigr),\qquad r\to\infty.\label{eq:Wong}
\end{equation}
\end{Theorem}

\begin{Remark}\label{rem:c_0 ne 0}It is natural to assume $c_{0}\ne0$ as in $(\Phi_{2})$
and we adopt it as a convention. See~\eqref{eq:1230c}. If one wants
to study the case $\varphi(s)\sim0$, one has only to consider the
difference of two functions with the same nontrivial expansion.
\end{Remark}

\section{Asymptotic expansion and invertibility\label{sec:Main}}

First, we utilize Wong's results reviewed in the previous section
to consider the contribution of the singularities at $s=0$ to Hankel
transforms. We consider an infinitely differentiable function $\varphi(s)$
and multiply it by a cut-off function. Therefore, $m$ in $(\Phi_{2})$
can be arbitrary and the conditions of uniform convergence are satisfied.
The $k$-th coefficient in the right-hand side in \eqref{eq:Wong}
vanishes if and only if $c_{k}=0$ or $\frac{1}{2}(\mu+k-\nu-1)$
is a nonnegative integer. Notice that~${\frac{1}{2}(\mu+k+\nu+1)}$
cannot be a pole of the Gamma function since $\operatorname{Re}(\mu+k+\nu+1)>k\ge0$.
This observation motivates us to introduce the set $K$ in the proposition
below.
\begin{Proposition}
\label{prop:Wong76a}Let $\mu, \nu\in\mathbb{C}$ and $\varphi(s)$
be an infinitely differentiable function in $(0,1)$. We make the
following assumptions:
\begin{align}
 & \operatorname{Re}(\mu+\nu)>-1,\label{eq:1230a}\\
 & \varphi^{(j)}(s)\sim\sum_{k=0}^{\infty}c_{k}\frac{d^{j}}{ds^{j}}s^{\mu+k},\qquad s\to+0;\quad j=0,1,2,\dots,\label{eq:1230b}\\
 & c_{0}\ne0.\label{eq:1230c}
\end{align}

Moreover, let $\chi_{0}(s)$ be an infinitely differentiable function
such that $\chi_{0}(s)=1$ in $0<s\le\varepsilon$ and $\chi_{0}(s)=0$
in $1-\varepsilon\le s<1$, where $0<\varepsilon<1/3$. We define
the set $K=K (\mu,\nu, \{ c_{k} \} _{k} )$ by
\begin{align}
K & =K(\mu,\nu,\{ c_{k}\} _{k})
 =\left\{ k\in\mathbb{N}_{0};\,c_{k}\ne0,\,\frac{1}{2}(\mu+k-\nu-1)\not\in\mathbb{N}_{0}\right\},\label{eq:set K}
\end{align}
where $\mathbb{N}_{0}$ is the set of nonnegative integers. Then,
we have the following.

If $K\ne\varnothing$,
\begin{align*}
 & \int_{0}^{1}\chi_{0}(s)\varphi(s)J_{\nu}(rs){\rm d}s=c_{k_{0}}\frac{\Gamma\left(\frac{1}{2}(\mu+k_{0}+\nu+1)\right)2^{\mu+k_{0}}}{\Gamma\left(\frac{1}{2}(-\mu-k_{0}+\nu+1)\right)r^{\mu+k_{0}+1}}+o\big(r^{-\operatorname{Re}(\mu+k_{0}+1)}\big)
\end{align*}
as $r\to\infty$, where $k_{0}=\min K$.

If $K=\varnothing$,
\begin{equation}
\int_{0}^{1}\chi_{0}(s)\varphi(s)J_{\nu}(rs){\rm d}s=o\big(r^{-A}\big)\label{eq:infinitely rapidly}
\end{equation}
as $r\to\infty$, where $A$ is an arbitrarily large real number.
\end{Proposition}

\begin{Proposition}
\label{prop:Wong76b}Assume $\operatorname{Re}(\nu)>-1$. If $\varphi(s)$
is an infinitely differentiable function in $(0,\infty)$, its support
is bounded and $\varphi(s)=s^{\nu+1}\big(1-s^{2}\big)^{\alpha}$, $\alpha\in\mathbb{C}$
near $s=0$, then
\[
\int_{0}^{\infty}\varphi(s)J_{\nu}(rs){\rm d}s=o\big(r^{-A}\big),
\]
where $A$ is an arbitrarily large real number.
\end{Proposition}

\begin{proof}
The set $K$ is empty in this case, since we have $\mu=\nu+1$ and
$k$ is even if $c_{k}\ne0$.
\end{proof}

Next we consider the contribution of the singularities at $s=1$
to finite Hankel transforms.
\begin{Proposition}
\label{prop:1112}Let $N$ be a nonnegative integer and $\Lambda$
be a complex number with $\operatorname{Re}\Lambda\ge N$. Assume that
$\psi(t)$ is an infinitely differentiable function in $0<t<1$ such
that $\psi(t)=0$ in~${0<t<\varepsilon}$ and that $\psi^{(k)}(t)$
is integrable for $0\le k\le N$. Set $\phi(t)=(1-t)^{\Lambda}\psi(t)$.
Then $\phi\big(s^{2}\big)$ is an infinitely differentiable function in $0<s<1$
and if $\nu\ge-N$, we have
\[
\int_{0}^{1}s^{\nu+1}\phi\big(s^{2}\big)J_{\nu}(rs){\rm d}s=o\big(r^{-(N+1/2)}\big),\qquad \text{as}\quad r\to\infty.
\]
\end{Proposition}

\begin{proof}
We have
\[
\phi^{(k)}(t)=\sum_{j=0}^{k}\binom{k}{j}\Lambda(\Lambda-1)\cdots(\Lambda-j+1)(-1)^{j}(1-t)^{\Lambda-j}\psi^{(k-j)}(t).
\]
Assume $N\ge1$. We have $\text{\ensuremath{\phi^{(k)}(1)=0}}$ if
$k\le N-1$. Set
\[
I_{k}=\int_{0}^{1}s^{\nu+k+1}\phi^{(k)}\big(s^{2}\big)J_{\nu+k}(rs){\rm d}s,\qquad k=0,1,2,\dots,N.
\]
The recurrence relation $\bigl(z^{\nu+1}J_{\nu+1}(z)\bigr)'=z^{\nu+1}J_{\nu}(z)$
(see, for example, \cite[formula~(10.6.6)]{DLMF}) yields
\[
\frac{{\rm d}}{{\rm d}s}\bigl\{ s^{\nu+1}J_{\nu+1}(rs)\bigr\} =rs^{\nu+1}J_{\nu}(rs).\label{eq:ladder operator}
\]
Combining this formula and integration by parts, we get
\begin{align}
I_{k} & =\frac{1}{r}\int_{0}^{1}\phi^{(k)}\big(s^{2}\big)\cdot rs^{\nu+k+1}J_{\nu+k}(rs){\rm d}s\nonumber \\ &=\frac{1}{r}\int_{0}^{1}\phi^{(k)}\big(s^{2}\big)\frac{{\rm d}}{{\rm d}s}\big\{ s^{\nu+k+1}J_{\nu+k+1}(rs)\big\} {\rm d}s
 =-\frac{2}{r}I_{k+1}\label{eq:recurrence}
\end{align}
if $k\le N-1$. We obtain
\begin{equation}
I_{0}=\left(-\frac{2}{r}\right)^{N}I_{N}\label{eq:1112a}
\end{equation}
for $N\ge1$. Notice that \eqref{eq:1112a} holds for $N=0$ as well.

On the other hand, it is well known that
\begin{align}
J_{\alpha}(z) & =\frac{2^{1/2}}{\pi^{1/2}}z^{-1/2}\cos\left(z-\frac{\alpha\pi}{2}-\frac{\pi}{4}\right)+O\big(z^{-3/2}\big)\label{eq:besselasymp J}
\end{align}
as $\mathbb{R}\ni z\to\infty$ (see, for example, \cite[formula~(10.17.3)]{DLMF}).
By \eqref{eq:besselasymp J} and the boundedness of $J_{\nu+N}(z)$
as~${z\to+0}$, there exists a bounded function~$R(z)$ $(0<z<\infty)$
such that
\[
J_{\nu+N}(z)=\frac{2^{1/2}}{\pi^{1/2}}z^{-1/2}\cos\left(z-\frac{(\nu+N)\pi}{2}-\frac{\pi}{4}\right)+z^{-3/2}R(z).\label{eq:Bessel asymp}
\]
We have
\begin{align*}
I_{N} ={}&\frac{2^{1/2}}{\pi^{1/2}}r^{-1/2}\int_{0}^{1}\phi^{(N)}\big(s^{2}\big)s^{\nu+N+1/2}\cos\left(rs-\frac{(\nu+N)\pi}{2}-\frac{\pi}{4}\right){\rm d}s\\
 & +r^{-3/2}\int_{0}^{1}\phi^{(N)}\big(s^{2}\big)s^{\nu+N-1/2}R(rs){\rm d}s.
\end{align*}
The first term in the right hand side is of order \smash{$o\big(r^{-1/2}\big)$} as
$r\to\infty$ by the Riemann--Lebesgue lemma. The second term is of
order \smash{$O\big(r^{-3/2}\big)$} by the boundedness of $R$. We have shown \smash{$I_{N}=o\big(r^{-1/2}\big)$}
and the combination of it and \eqref{eq:1112a} gives \smash{$I_{0}=o\big(r^{-(N+1/2)}\big)$}.
\end{proof}

\begin{Proposition}\label{prop:1112b}Let $N$ be a nonnegative integer and $\nu$,
$\lambda_{0},\dots,\lambda_{m}$, $\Lambda$, $a_{0},\dots, a_{m}$ be complex numbers. Assume $\operatorname{Re}\nu>-1$,
$-1<\operatorname{Re}\lambda_{0}<\operatorname{Re}\lambda_{1}<\cdots<\operatorname{Re} \lambda_{m}<\operatorname{Re} \Lambda$,
$N\le\operatorname{Re} \Lambda$, $\operatorname{Re} \lambda_{0}\le N-1$. Assume
that the function $\phi(t)$ in $0<t<1$ satisfies
\[
\phi(t)=\sum_{k=0}^{m}a_{k}(1-t)^{\lambda_{k}}+(1-t)^{\Lambda}\psi(t),\qquad 1-2\varepsilon<t<1,
\]
where $a_{0}\ne0$ and $\psi(t)$ is an infinitely differentiable
function in $0<t<1$ such that $\psi^{(k)}(t)$ is integrable in $1-2\varepsilon<t<1$
for $0\le k\le N$.

Let $\chi_{1}(t)$ be an infinitely differentiable function such that
$\chi_{1}(t)=0$ in $0<t\le1-2\varepsilon$ and~${\chi_{1}(t)=1}$ in
$1-\varepsilon\le t<1$. Set\footnote{Notice that the cutoff is incomplete in the sense that $\chi_{1}(t)$
does not appear in the summation term.}
\[
\tilde{\phi}(t)=\sum_{k=0}^{m}a_{k}(1-t)^{\lambda_{k}}+\chi_{1}(t)(1-t)^{\Lambda}\psi(t),\qquad 0<t<1.
\]
Then as $r\to\infty$,
\begin{align*}
 \int_{0}^{1}s^{\nu+1}\tilde{\phi}\big(s^{2}\big)J_{\nu}(rs){\rm d}s
 ={}&a_{0}\frac{2^{\lambda_{0}+1/2}}{\pi^{1/2}}\Gamma(\lambda_{0}+1)r^{-(\lambda_{0}+3/2)}\cos\left(r-\frac{\pi}{2}(\nu+\lambda_{0}+1)-\frac{\pi}{4}\right)\\
 & +o\big(r^{-\operatorname{Re}(\lambda_{0}+3/2)}\big).
\end{align*}
\end{Proposition}

\begin{proof}
By \cite[formula~(26(33)a)]{Erdelyi-Table} or \cite[formula~(6.567.1)]{GR}, we have
\begin{align}
\int_{0}^{1}s^{\nu+1}\bigl(1-s^{2}\bigr)^{\alpha}J_{\nu}(rs){\rm d}s & =2^{\alpha}\Gamma(\alpha+1)r^{-(\alpha+1)}J_{\nu+\alpha+1}(r),\label{eq:GR6.567.1}
\end{align}
if $r>0$, $\operatorname{Re}\nu>-1$, $\operatorname{Re}\alpha>-1$.\footnote{If we set $s=\sin\theta$, we get Sonine's first finite integral.
Two methods of its evaluation are given in \cite[formula~(12.11)\,(1)]{Watson}).
There is another proof of \eqref{eq:GR6.567.1}. When $\alpha$ is
a nonnegative integer, we can prove it following~\eqref{eq:recurrence}
in the proof of Proposition~\ref{prop:1112}. To generalize it to
$\operatorname{Re}\alpha>-1$, we divide the equality by $\Gamma(\alpha+1)$
and apply the Carlson's theorem \cite[p.~110]{AAR}. Boundedness
is guaranteed by Poisson's integral representation of Bessel functions~\cite[formula~(10.9.4)]{DLMF}.} We employ this formula and Proposition~\ref{prop:1112} to obtain
\begin{align*}
\int_{0}^{1}s^{\nu+1}\tilde{\phi}\big(s^{2}\big)J_{\nu}(rs){\rm d}s= \sum_{k=0}^{m}a_{k}2^{\lambda_{k}}\Gamma(\lambda_{k}+1)r^{-(\lambda_{k}+1)}J_{\nu+\lambda_{k}+1}(r) +o\big(r^{-(N+1/2)}\big).
\end{align*}
We finish the proof by using \eqref{eq:besselasymp J}.
\end{proof}

Now we give our main results.
\begin{Theorem}
\label{thm:smooth}Let $n\ge2$ be an integer. Let $\varphi(s)$ be
an infinitely differentiable function in the open interval $(0,1)$
and $\phi(t)$ $(0<t<1)$ be the function such that
\[
\varphi(s)=s^{n/2}\phi\big(s^{2}\big),\qquad0<s<1.
\]
Set $\nu=n/2-1$. If $\mu$ satisfies \eqref{eq:1230a}, $\varphi(s)$
satisfies \eqref{eq:1230b}, \eqref{eq:1230c} and $\phi(t)$ satisfies
the conditions of Proposition {\rm \ref{prop:1112b}}, then
\begin{equation}
f(x)=|x|^{-n/2}\varphi(|x|)\chi_{[0,1]}(|x|)=\phi\bigl(|x|^{2}\bigr)\chi_{[0,1]}(|x|),\qquad x\in\mathbb{R}^{n}\label{eq:f(x)}
\end{equation}
is an invertible distribution. Here $\chi_{[0,1]}(\cdot)$ is the
indicator function of the interval $[0,1]$.
\end{Theorem}

\begin{proof}
Set $\nu=\nu(n)=n/2-1$ and
\[
\tilde{\varphi}(s)=\chi_{0}(s)\left\{ \varphi(s)-s^{n/2}\sum_{k=0}^{m}a_{k}\bigl(1-s^{2}\bigr)^{\lambda_{k}}\right\} .
\]
Notice that
\[
{\displaystyle \tilde{\varphi}{}^{(j)}(s)\sim\sum_{k=0}^{\infty}c_{k}\frac{{\rm d}^{j}}{{\rm d}s^{j}}s^{\mu+k}+\sum_{\ell=0}^{\infty}A_{\ell}\frac{{\rm d}^{j}}{{\rm d}s^{j}}s^{n/2+2\ell},\qquad s\to+0}
\]
for some $A_{\ell}$. The second sum, which corresponds to $s^{n/2}\sum_{k=0}^{m}a_{k}\big(1-s^{2}\big)^{\lambda_{k}}$,
does not contribute to the finite Hankel transform because of Proposition
\ref{prop:Wong76b}.

Let $K_{n}=K (\mu,\nu(n), \{ c_{k} \} _{k} )$ be
the set defined by \eqref{eq:set K} with $\nu=\nu(n)$. If $K_{n}\ne\varnothing$,
we have
\begin{align*}
H_{0} & :=\int_{0}^{1}\tilde{\varphi}(s)J_{\nu(n)}(rs){\rm d}s =c_{k_{0}}\frac{\Gamma\left(\frac{1}{2}(\mu+k_{0}+\nu(n)+1)\right)2^{\mu+k_{0}}}{\Gamma\left(\frac{1}{2}(-\mu-k_{0}+\nu(n)+1)\right)r^{\mu+k_{0}+1}} +o\big(r^{-\operatorname{Re}(\mu+k_{0}+1)}\big),
\end{align*}
and if $K_{n}=\varnothing$, $H_{0}$ decreases rapidly of arbitrary
order by~\eqref{eq:infinitely rapidly}. On the other hand, for $\tilde{\phi}$
defined in Proposition~\ref{prop:1112b}, we have
\begin{align*}
H_{1}:={} & \int_{0}^{1}s^{\nu(n)+1}\tilde{\phi}\big(s^{2}\big)J_{\nu(n)}(rs){\rm d}s\\
={} & a_{0}\frac{2^{\lambda_{0}+1/2}}{\pi^{1/2}}\Gamma(\lambda_{0}+1)r^{-(\lambda_{0}+3/2)}\cos\left(r-\frac{\pi}{2}(\nu(n)+\lambda_{0}+1)-\frac{\pi}{4}\right)
 +o\big(r^{-\operatorname{Re}(\lambda_{0}+3/2)}\big).
\end{align*}
We can prove that $H_{0}+H_{1}$ satisfies \eqref{eq:slowly decreasing q}
in the following way.

If $K_{n}=\varnothing$, $H_{1}$ is dominant. We assume $K_{n}\ne\varnothing$
from now on. If $\operatorname{Re}\mu+k_{0}+1<{\operatorname{Re} \lambda_{0}+3/2}$,
then $H_{0}$ is dominant. If $\operatorname{Re}\mu+k_{0}+1>\operatorname{Re} \lambda_{0}+3/2$,
then $H_{1}$ is dominant. If ${\operatorname{Re}\mu+k_{0}+1}={\operatorname{Re} \lambda_{0}+3/2}$,
we can pick up those $r$'s for which the cosine factor vanishes and
the contribution of~$H_{1}$ becomes irrelevant. In any case, $H_{0}+H_{1}$
satisfies~\eqref{eq:slowly decreasing q}. By Theorem~\ref{thm:Stein-Weiss}
and Proposition~\ref{prop:q single variable}, $|x|^{-n/2}\tilde{\varphi}(|x|)+\tilde{\phi}\bigl(|x|^{2}\bigr)$
is invertible.

Notice that $\varphi(s)-\tilde{\varphi}(s)-s^{n/2}\tilde{\phi}\big(s^{2}\big)$
vanishes near $s=0,1$ and
\begin{align*}
 & f(x)-|x|^{-n/2}\tilde{\varphi}(|x|)-\tilde{\phi}\bigl(|x|^{2}\bigr)\in\mathcal{C}_{0}^{\infty}(\mathbb{R}^{n}).
\end{align*}
By Proposition \ref{prop:u*v}\,(v), $f(x)$ is invertible.
\end{proof}
\begin{Example}
If $\lambda>-n/2$, $\rho>0$, then
$
f(x)=|x|^{\lambda-n/2}\bigl(1-|x|^{2}\bigr)^{\rho-1}\chi_{[0,1]}(|x|)
$
is an invertible distribution. This corresponds to the case of $\varphi(s)=s^{\lambda}\bigl(1-s^{2}\bigr)^{\rho-1}$
and $\phi(t)=t^{\lambda/2-n/4} (1-t )^{\rho-1}$.

The finite Hankel transform $\int_{0}^{1}s^{\lambda}\bigl(1-s^{2}\bigr)^{\rho-1}J_{\nu}(rs){\rm d}s$
can be written in terms of the generalized hypergeometric function
$_{2}F_{3}$ by \cite[formula~(6.569)]{GR} and it is good enough to prove invertibility.
The advantage of our method is that it is stable under small perturbations
and works even if no closed form expression is available.
\end{Example}

\begin{Remark}\label{rem:0 expansion}In \eqref{eq:1230b}, $j$ can be arbitrarily
large. This assumption can be relaxed in some cases. We give such
an example. Assume
\[
\varphi(s)\sim0=\sum_{k=0}^{\infty}0\cdot s^{\mu+k}
\]
as $s\to+0$, where $-n/2<\operatorname{Re}\mu\le-1$. We do not assume
that term by term differentiation is possible:\footnote{For example, $\varphi(s)={\rm e}^{-1/s}\sin {\rm e}^{1/s}$.}
 $m$ in $(\Phi_{1})$ is 0. Moreover, we have removed the assumption
$c_{0}\ne0$ and do not need the set~$K_{n}$. We keep all the other
assumptions of Theorem \ref{thm:smooth} unchanged. By~\eqref{eq:Wong}
and Remark~\ref{rem:c_0 ne 0}, we have $H_{0}=o(1)$. If $\operatorname{Re}\lambda_{0}\le-3/2$,
$H_{1}$ is dominant and $f(x)$ defined by \eqref{eq:f(x)} is invertible.
\end{Remark}

\begin{Theorem}
\label{thm:smooth2}Let $n\ge2$ be an integer and $N$ be a nonnegative
integer. Let $\varphi(s)$ be an infinitely differentiable function
in the open interval $(0,1)$ and $\phi(t)\,(0<t<1)$ be the function
such that~$\varphi(s)=s^{n/2}\phi\bigl(s^{2}\bigr)$. Set $\nu=n/2-1$. Suppose
that $\mu$ satisfies \eqref{eq:1230a} and that $\varphi(s)$ satisfies
\eqref{eq:1230b}, \eqref{eq:1230c}. Moreover, suppose $\operatorname{Re}\Lambda\ge N$
and
\[
\phi(t)=(1-t)^{\Lambda}\psi(t),\qquad 1-2\varepsilon<t<1,
\]
where $\psi(t)$ is an infinitely differentiable function in $0<t<1$
such that $\psi^{(k)}(t)$ is integrable in~${1-2\varepsilon<t<1}$
for $0\le k\le N$.

If $K_{n}=K (\mu,n/2-1, \{ c_{k} \} _{k} )\ne\varnothing$
and $\operatorname{Re}(\mu+k_{0}+1/2)\le N$, then
\[
f(x)=|x|^{-n/2}\varphi(|x|)\chi_{[0,1]}(|x|)=\phi\bigl(|x|^{2}\bigr)\chi_{[0,1]}(|x|),\qquad x\in\mathbb{R}^{n}
\]
is an invertible distribution.
\end{Theorem}

\begin{proof}
We have
\begin{align*}
H_{0} =c_{k_{0}}\frac{\Gamma\left(\frac{1}{2}(\mu+k_{0}+\nu(n)+1)\right)2^{\mu+k_{0}}}{\Gamma\left(\frac{1}{2}(-\mu-k_{0} +\nu(n)+1)\right)r^{\mu+k_{0}+1}}+o\big(r^{-\operatorname{Re}(\mu+k_{0}+1)}\big),\qquad
H_{1} =o\big(r^{-(N+1/2)}\big)
\end{align*}
by Propositions \ref{prop:Wong76a} and \ref{prop:1112}. Since $\operatorname{Re}(\mu+k_{0}+1)\le N+1/2$,
$H_{0}$ is dominant and $H_{1}$ is negligible.
\end{proof}
\begin{Theorem}
\label{thm:smooth3}Let $n\ge2$ be an integer. Let $\varphi(s)$
be an infinitely differentiable function in $s>0$ such that $\varphi(s)=0$
in $s\ge1$. Set $\nu=n/2-1$. Suppose that $\mu$ satisfies \eqref{eq:1230a}
and that $\varphi(s)$ satisfies \eqref{eq:1230b}, \eqref{eq:1230c}.
Then
\[
f(x)=|x|^{-n/2}\varphi(|x|)\chi_{[0,1]}(|x|),\qquad x\in\mathbb{R}^{n}
\]
is an invertible distribution if and only if $K_{n}=K (\mu,n/2-1, \{ c_{k} \} _{k} )$
is non-empty.
\end{Theorem}

\begin{proof}
The quantity $\Lambda$ in Theorem~\ref{thm:smooth2} can be an arbitrarily
large integer and we have $H_{1}=o\bigl(r^{-A}\bigr)$, where~$A$ is arbitrarily
large. By Proposition~\ref{prop:Wong76a}, the invertibility of~$f(x)$
is equivalent to the non-emptiness of~$K_{n}$.
\end{proof}
\begin{Remark}
We can find a lot of invertible distributions by combining Theorems~\ref{thm:smooth}, \ref{thm:smooth2} and~\ref{thm:smooth3} with
Remark \ref{rem:invertible distributions} and Proposition \ref{prop:u*v}.
\end{Remark}

\begin{Remark}
We can formulate an invertibility theorem by using \cite{Wong1977}
instead of~\cite{Wong1976}. In~\cite{Wong1977}, the function $\varphi(s)$
is allowed to have an asymptotic expansion involving powers of logarithms.
\end{Remark}

\subsection*{Acknowledgements}
The authors would like to thank the anonymous referees who provided useful comments.
The first author is supported by JSPS KAKENHI Grant Number 21K03265.

\vspace{-1mm}

\pdfbookmark[1]{References}{ref}
\LastPageEnding

\end{document}